\documentclass[a4paper,11pt]{article}
\usepackage{amssymb,amsmath,amsthm}
\usepackage[UKenglish]{babel}
\usepackage{microtype}
\usepackage[ebgaramond,lf]{newtx}
\usepackage{a4wide}
\usepackage{anyfontsize}
\usepackage{complexity}
\usepackage[dvipsnames]{xcolor}
\usepackage{hyperref} 
\usepackage[nameinlink]{cleveref} 
\hypersetup{colorlinks=true,citecolor=RoyalBlue,linkcolor=RoyalBlue,urlcolor=RoyalBlue}

\theoremstyle{plain}
\newtheorem{theorem}{Theorem}
\newtheorem{lemma}[theorem]{Lemma}

\theoremstyle{definition}
\newtheorem{remark}[theorem]{Remark}

\begin{document}

\author{Johan H{\aa}stad\\
KTH\\
\texttt{johanh@kth.se}
\and
Bj{\"o}rn Martinsson\\
KTH\\
\texttt{bmart@kth.se}
\and
Tamio-Vesa Nakajima\\
University of Oxford\\
\texttt{tamio-vesa.nakajima@cs.ox.ac.uk}
\and
Stanislav \v{Z}ivn\'y\\
University of Oxford\\
\texttt{standa.zivny@cs.ox.ac.uk}
}

\title{A logarithmic approximation of linearly ordered colourings\thanks{This
work was supported by UKRI EP/X024431/1, by a Clarendon Fund Scholarship and
by the Knut and Alice Wallenberg Foundation. For
the purpose of Open Access, the authors have applied a CC BY public copyright
licence to any Author Accepted Manuscript version arising from this submission.
All data is provided in full in the results section of this paper.}}

\date{}
\maketitle

\begin{abstract}
    A linearly ordered (LO) $k$-colouring of a hypergraph 
    assigns to each vertex a colour from the set $\{0,1,\ldots,k-1\}$ in such a way that
    each hyperedge has a unique maximum element. Barto, Batistelli, and Berg
    conjectured that it is \NP-hard to find an LO $k$-colouring of an LO
    2-colourable 3-uniform hypergraph for any constant $k\geq 2$~[STACS'21] but even the case
    $k=3$ is still open. Nakajima and \v{Z}ivn\'{y} gave 
    polynomial-time algorithms for finding, given an LO 2-colourable 3-uniform
    hypergraph, an LO colouring with $O^*(\sqrt{n})$ colours~[ICALP'22] and an
    LO colouring with $O^*(\sqrt[3]{n})$ colours~[ACM ToCT'23]. Very recently,
    Louis, Newman, and Ray gave an SDP-based algorithm with $O^*(\sqrt[5]{n})$
    colours~[FSTTCS'24]. We present two
    simple polynomial-time algorithms that find an LO colouring with
    $O(\log_2(n))$ colours, which is an exponential improvement.
\end{abstract}

\section{Introduction}

Given a graph $G$, the \emph{graph $k$-colouring} problem asks
to find a colouring of the vertices of $G$ by colours from the set
$\{0,1,\ldots,k-1\}$ in such a way that no edge is monochromatic. 
The \emph{approximate graph colouring problem} asks, given a $k$-colourable
graph $G$, to find an $\ell$-colouring of $G$, where $\ell\geq k$. For $k=3$, the
state-of-the-art results are \NP-hardness of the case $\ell=5$~\cite{BBKO21} and
a polynomial-time algorithm for finding a colouring with $\ell=O(n^{0.19747})$ colours,
where $n$ is the number of vertices of the input graph $G$~\cite{KTY24:stoc}.
For non-monochromatic colourings of hypergraphs, it is known that finding an
$\ell$-colouring of a $k$-colourable $r$-uniform hypergraph is \NP-hard for any
constant $\ell\geq k\geq 2$ and $r\geq 3$~\cite{DRS05}, 
and also some positive results are known for colourings with
super-constantly many colours,
e.g.~\cite{Krivelevich03:ja,Krivelevich01:ja,Chlamtac08:approx}.

A new promise hypergraph colouring problem was identified
in~\cite{Barto21:stacs}.
Given a 3-uniform hypergraph $H$, a colouring of the
vertices of $H$ with colours from the set $\{0,1,\ldots,k-1\}$ is called
a \emph{linearly ordered} (LO) $k$-colouring if every edge $e$ of $H$ satisfies the following: if two vertices of $e$ have the same colour then the third colour is
larger. More generally, a colouring of a hypergraph $H$ is an LO
colouring if every edge of $H$ has a unique maximum colour. (Note that the
two definitions coincide for 3-uniform hypergraphs.)
Barto et al. conjectured that finding an LO
$\ell$-colouring of a 3-uniform hypergraph that admits an LO $k$-colouring is
\NP-hard for every constant $\ell\geq k\geq 2$~\cite{Barto21:stacs} but even the
case $k=2$ and $\ell=3$ is open. Nakajima and \v{Z}ivn\'y established
\NP-hardness for some regimes of the parameters $k,\ell,
r$~\cite{NZ22:icalp,NZ23:toct}  (where $r$ is the uniformity of the input hypergraph) and, very recently, Filakovk\'y et al.~\cite{fnotw24:stacs} showed \NP-hardness of
the case $k=3$, $\ell=4$, $r=3$.
More importantly for this paper, Nakajima and \v{Z}ivn\'y also considered
finding an LO $f(n)$-colouring of an LO 2-colourable 3-uniform hypergraph
with $n$ vertices and presented polynomial-time algorithms with
$f(n)=O(\sqrt{n \log \log n}/\log n)$~\cite{NZ22:icalp}
and
$f(n)=O(\sqrt[3]{n \log \log n / \log n})$~\cite{NZ23:toct}.
Very recently, Louis, Newman, and Ray~\cite{LNR24:arxiv} have given a
polynomial-time SDP-based algorithm with $f(n)=O^*(\sqrt[5]{n})$ colours.

As our main result, we improve their results by an exponential factor. 

\begin{theorem}\label{thm:main}
    There is an algorithm which, if given a 3-uniform hypergraph $H$ with $n \geq 4$ vertices
    and $m$ edges that admits an LO 2-colouring, finds an LO
    $\log_2(n)$-colouring of $H$ in time $O(n^3 + nm)$.
\end{theorem}

In fact we present two different algorithms that return colourings using $O(\log n)$ colours.  Both 
are based on solving the natural system of linear equations implied by the existence of an LO 2-colouring.  In one case, the system is solved modulo 2, and in the other case, the system is solved over the rational numbers.

While the $H$ which we are given as input is 3-uniform, we will need the notion we define next in greater generality; hence we define it for general hypergraphs. For each edge $\{x_1, \ldots, x_r\}$ of $H$, we write an equation $v_{x_1} + \cdots + v_{x_r} = 1$ where we initially use equality modulo 2 but as stated above we later use the same
system over the rational numbers.   Let $A$ be this set of equations, written as a matrix with $m$
    rows and $n$ columns. (Note that $A$ is the \emph{incidence matrix} of $H$.) Thus $v$ is a solution if and only if
    $Av=1^m$.  Clearly a valid LO 2-colouring gives one solution but in the general case, the system 
    has a large dimensional affine space as its set of solutions and the 
    desired solution is hard to find.

    This is the journal version of the conference paper~\cite{HMNZ24:approx}. The main change is that we have a new method for using the solution over the rational numbers. This improvement reduces the number of colours used by the algorithm in \Cref{sec:algoQ} from $2 \log_2 n$ to $1.5 \log_2 n$ colours.

\paragraph{Related work.} While the notion of LO colouring was first identified in the context of promise problems in~\cite{Barto21:stacs}, it is identical to the notion of \emph{unique-maximum colouring}~\cite{Cheilaris13:sidma} of a hypergraph, and similar to that of \emph{conflict-free colouring}, introduced by Even et al.~\cite{Even02} and Smorodinsky~\cite{smorodinskythesis}. (In a conflict-free colouring, every hyperedge must contain a unique value, but this value need not be the maximum). There is also the related notion of a \emph{graph unique-maximum colouring}, also known as \emph{ordered colouring} or \emph{vertex ranking}~\cite{Cheilaris11:jda}, which is the same as a unique-maximum colouring of the path hypergraph of a graph (i.e.~the hypergraph whose edges are paths of the graph). We refer the reader to~\cite{Smorodinsky2013} for more on conflict-free and unique-maximum colouring.

It is worth mentioning Smorodinsky's framework for unique-maximum colouring of a hypergraph (see \cite{Smorodinsky07} where it is proposed for conflict-free colouring, and also~\cite{Smorodinsky2013} where it is extended to unique-maximum colouring), which does the following: Given a hypergraph $H$, find a non-monochromatic colouring with few colours, select the largest cardinality colour class, colour it in our unique-maximum colouring with the minimum colour, then continue recursively.  Smorodinsky uses this algorithm to find unique-maximum colourings of graphs where it is \emph{hereditarily} easy to find non-monochromatic colourings with few colours --- in particular graphs that come from geometric situations. Unfortunately in our case this does not happen. Consider for example the hypergraph $H$ with vertices $v_i$ for $i \in [k]$ and $w_{ij}$ for $i, j \in [k]$; and edges $(v_i, v_j, w_{ij})$. Suppose we apply Smorodinsky's algorithm to this hypergraph, and at the first step we find the non-monochromatic colouring given by $v_i \mapsto 0$ and $w_{ij} \mapsto 1$. Then at the next step Smorodinsky's algorithm must colour the clique on $v_1, \ldots, v_k$ --- this basically means that it outputs a colouring with $\Theta(\sqrt{n})$ colours, since the number of vertices is $\Theta(k^2)$.\footnote{Observe that even the linearisation trick from~\cite{NZ23:toct} does nothing for this hypergraph, as it is already linear --- i.e.~every two edges intersect in at most one vertex. The essence of the trick is to identify vertices $x$ and $y$ if there exist vertices $a, b$ and edges $(x, a, b)$ and $(y, a, b)$, as such vertices must have the same colour in any LO 2-colouring.} Thus even if we ignore the fact that finding the colouring on the graph that we get after the first step is \NP-hard in general (note that we could encode finding a colouring of any graph we wanted rather than just the clique), we do not necessarily get a small number of colours by this framework. The algorithm for conflict-free colourings presented in~\cite{Even02} is similar. In each step, it chooses a set of vertices that intersects each edge $e$ either (i) in one vertex, or (ii) in $< |e|$ vertices. (These two cases are disjoint only when $|e| = 1$.) It then colours these vertices with one colour, throwing away all edges coloured by exactly one vertex. To generalise this to unique-maximum colourings one would need to keep all edges except those for which exactly one vertex remains --- thus this algorithm does not work in our setting for the same reason.

While these algorithms have some similarity with ours, we critically do not find a non-monochromatic colouring of our hypergraph $H$ at each step. Indeed, what we do at any particular step either immediately solves an edge or leaves it completely intact --- this lets us keep the property of LO 2-colourability for what is left to colour intact.

Let us return to the world of promise problems. Our problem is a \emph{promise CSP}~\cite{BBKO21} of the form ``given a 3-uniform LO 2-colourable hypergraph, find a homomorphism to a fixed 3-uniform hypergraph $H$''. 
In~\cite{Barto21:stacs}, the computational complexity of this problem was classified for all 3-vertex hypergraphs $H$ except for the $H$ that encodes LO 3-colouring. This gap in their findings is what motivated the authors to introduce the linearly ordered colouring problem. In~\cite{CKKNZ25:TOCL}, the authors classified the problem for any $H$ whose edges do not contain 3 distinct elements; i.e.~all edges are of the form $(x, x, y)$. In particular, this implies that the \emph{rainbow-free}\footnote{In the \emph{rainbow-free} variant, the goal is to find a colouring without rainbow edges, i.e.~no edge can contain 3 distinct colours.} LO 2- vs.~LO $k$-colouring problem is \NP-hard for every fixed $k$. However, note that the case of $k=3$ was already shown to be NP-hard by \cite{Barto21:stacs}. 

\section{Algorithm based on equations modulo 2}\label{sec:algomod2}

In this section all linear equations are taken modulo 2. For the following, given a set $S$ of positive integers, an $S$-uniform hypergraph is a hypergraph where all edges have sizes taken from $S$. We first prove the following subprocedure of the main algorithm.

\begin{lemma}\label{lem:innerstep}
    There is an algorithm which, if given a $\{2, 3\}$-uniform hypergraph $H$ with $n$ vertices and
    $m$ edges that admits an LO 2-colouring and such that the implied linear
    system of equations $Av=1^m$ does not fix the value of any variable, outputs a subset $T$ of vertices that intersects edges of size three in zero or two vertices and edges of size two in exactly one vertex.
    Moreover, we have  $\vert T\vert \geq n / 2$. The algorithm runs in $O(n^3 + nm)$ time.
\end{lemma}

\begin{proof}

    We first describe a randomised version of our algorithm, and then derandomise
    it. The set of solutions to $Av=1^m$ is an affine space and hence a generic solution can be written as $v=v^0+\sum_{i=1}^r a_i v^i$ for a basic solution $v^0$, linearly independent solutions to the homogeneous system $v^i$, and field elements (in this case bits) $a_i$.
    The fact that no variable is fixed implies that for each vertex $x$ there is 
    some positive $i$ such that $v^i_x=1$.
    
     For the randomised algorithm choose $a_1, \ldots, a_r$ to be independent identically distributed uniformly random bits, and set $T$ to be 
     the set of variables $x$, such that $v_x=0$.  Clearly $T$ satisfies
     the conclusion of the lemma as in each edge we have an odd number of ones. 
     For every vertex $x$, there exists positive $i$ such that $v_x^i = 1$ --- hence, due to the influence of $a_i v_x^i$, we see that $x$ is included in $T$ with probability $\frac{1}{2}$.
     Thus, on average $T$ contains half the vertices.

    Now, we derandomise this algorithm using the method of conditional
    expectations. Go through the variables $a_i$ in increasing order and 
    fix its value once and for all.  Fixing the value of $a_i$ determines
    the value(s) of some $v_x$ while other values remain undetermined.  For each
    value being determined $v_x^i=1$ and hence one value of $a_i$ gives
    the final value 0 and the other gives final value 1.  Set $a_i$ such
    that at least half the determined values are 0.  After we have 
    fixed all $a_i$ this way, we have a final solution with at least
    $n/2$ zeroes.
    
    The bottleneck of the running time of this algorithm is solving the linear system of equations. This can be done in the advertised running time since every equation has $O(1)$ entries. 
\end{proof}

\begin{proof}[Proof of \Cref{thm:main}]
As a preliminary step, we eliminate any variable determined by the system $Av=1^m$. Note that if the colour of a vertex is determined by the system $Av=1^m$, then this vertex must have that same colour in \emph{all} LO 2-colourings. Fix these variables once and for all and eliminate them from the equation system.
For all vertices that have been given the colour $1$, we set the colours of the two other vertices in all of its edges to be $0$. This process of identifying fixed variables and eliminating them is then repeated until the system $Av=1^m$ contains no variables fixed to a constant. At any fixed point of this process, for every edge, either all vertices in that edge are fixed (and the edge has a unique-maximum as required), or exactly one vertex in it is fixed to 0. 

Now, remove all coloured vertices from the hypergraph $H$, shrinking the edges they belonged to. The remaining hypergraph will no longer be $3$-uniform (it will be $\{2, 3\}$-uniform though), but importantly it will still be LO 2-colourable. Our goal is still to LO colour the remaining hypergraph, since any edge partially coloured by the preliminary step above must have had exactly one vertex $v$ fixed to 0; and hence, if we LO colour the edge that resulted from removing $v$, this leads to an LO colouring of the original hypergraph when $v$ is assigned 0.

Consider the following algorithm, where $i$ starts at $0$.
\begin{enumerate}
    \item If the hypergraph $H$ has at most, say, 20  vertices, find an LO colouring of $H$ by brute force using colours $i$ and $i + 1$. (It exists since $H$ is LO 2-colourable.)
    \item Otherwise, find the subset $T$ guaranteed by \Cref{lem:innerstep}.
    \item Colour the vertices in $T$ by colour $i$. Remove the vertices in $T$ from $H$. Remove all edges that intersect $T$ from $H$. Increment $i$ by $1$.
    \item Repeat.
\end{enumerate}
Note that $| T| \geq n / 2$ and thus within $-4 +  \log_2 n$ repetitions we reach the first case. Each step adds one colour and we get two additional
colours from the final brute-force colouring for a total of at most $\log_2 n$ colours.
The output is correct as the first time some vertex in an edge is coloured, for edges with three vertices exactly one more vertex in the same  edge is coloured at the same time, and for edges with two vertices only that vertex is coloured at that time. The remaining vertex is given a higher colour and hence the edge is correctly coloured.

For the time complexity, we again note that it is dominated by the time needed to solve the linear system of equations. Since $n$ decreases by a factor of 2 at every step, and the cost of the inner loop is $O(n^3 + nm)$, this gives us the required time complexity --- even ignoring the fact that $m$ also is decreased.
\end{proof}

Note that the number 20 selected above can be increased to any number that is $O(\log n)$ and the algorithm remains polynomial time (since we must compute the colouring for a subgraph of this size by brute force). If we stop the algorithm at $B$ vertices, then we save $\log B + \Theta(1)$ colours, since this is how many colours the algorithm would have used to colour the last $B$ vertices. By setting $B = \Theta(\log n)$, we can thus save $\Theta(1) + \log \log n$ colours while keeping run time of the algorithm polynomial in $n$.

A slight variant can be obtained by instead counting the number of remaining edges with no coloured vertex. Once we have no more such vertices, we colour all remaining vertices with the next colour. For such edge, a random solution $v$ gives the four sets of
values $(0,0,1)$, $(0,1,0)$, $(1,0,0)$ and $(1,1,1)$ with equal probabilities. 
Thus the number of edges decreases, on average, by a factor 4 for each
iteration. (Note that all the edges of size 2 are solved in the first iteration, so there is no need to count them.)
It is easy to achieve this deterministically by conditional
expectations.  Once we have not edge with three uncoloured vertices we can colour all remaining uncoloured vertices with the next colour. We state the conclusion as a theorem.

\begin{theorem}\label{thm:main2}
    There is an algorithm which, if given a 3-uniform hypergraph $H$ with $n$ vertices
    and $m \geq 1$ edges that admits an LO 2-colouring, finds an LO
    $(2 + \frac 12 \log_2(m))$-colouring of $H$ in time $O(n^3 + nm)$.
\end{theorem}

\begin{remark}
  Our algorithm has some similarity with algorithms for
  \emph{temporal} CSPs~\cite{BK10:jacm}. Note that a rainbow-free LO $\omega$-colouring
  (which means an LO colouring, but with no restriction on the number of
  colours; also, \emph{rainbow tuples} i.e.~tuples $(x, y, z)$ with $x \neq y \neq z \neq x$ are disallowed) is a temporal CSP; to solve it, one finds a subset that could be the
  smallest colour (by solving mod-2 equations as above), sets that colour, then
  continues recursively. The difference is that for a rainbow-free LO $\omega$-colouring one
  does not care about the number of colours, so one can find any nonempty set of
  vertices to set the lowest colour to, whereas in our problem we are trying to
  find a large set of this kind. We note that the algorithm of~\cite{NZ23:toct} also uses this approach when setting ``small colours''.
\end{remark}

\begin{remark}
 We remark that the subprocedure of our algorithm computes the exclusive or
of two vectors of bits. Thus the algorithm runs very fast in practice ---
  on most architectures hundreds of operations of this kind are done at one time by (i) packing the bits within a larger word and (ii) using SIMD instructions.
\end{remark}

\section{\texorpdfstring{Algorithm using $\mathbb{Q}$}{Algorithm using Q}}
\label{sec:algoQ}

In this section we present a more complicated algorithm
which uses more colours.  This
might seem pointless, and indeed it might be.  On the
other hand the ideas used are slightly different and hence there
might be situations where the ideas of this section can
turn out to be useful.  It is also curious to see that we can use the same system
of linear equations, now over the rational numbers, in a
rather different way.  The algorithm here is in fact essentially
saying that we can always use the unbalanced case of~\cite{LNR24:arxiv}.   As this eliminates many complications
and in particular makes it possible to completely avoid
any semi-definite programming, we state all facts needed in
the current section rather than refer to the very similar
statements in ~\cite{LNR24:arxiv}.
As already stated, all arithmetic in this section is over
the rational numbers.  In this situation, no variables can
be determined as we can set $v^0$ to
have all coordinates equal to $1/3$.  

We study the homogeneous system $Av=0^m$ and
by the assumption of LO 2-colourability it has
a solution, $w$, with coordinates either $-\frac 13$ or $\frac 23$.
Let us first show how
solutions over the rational numbers can be used to find LO colourings.
This is the same lemma used in the unbalanced case of ~\cite{LNR24:arxiv}.

\begin{lemma}\label{lemma:unbal}
Suppose we have a solution, $u$, to the homogeneous system where $M$
is the maximal value of the absolute value of a coordinate
and $m > 0$ is the minimal absolute value.  Then, we can
LO colour with $2+\log_2{(M/m)}$ colours.
\end{lemma}

\begin{proof} For notational convenience let us instead require that the
minimal colour in each edge should be unique. We can simply reverse
the order of the colours at the end.  By scaling we can assume $M=1$.
We use even colours for positive coordinates and we
give $x$ the colour $2\ell$ if $v_x$ is at most $2^{-(2\ell-1)}$
and strictly larger than $2^{-(2\ell+1)}$.
For negative coordinates we use $2\ell+1$ as
the colour if $v_x$ is between $-2^{-2\ell}$ (inclusive)
and $-2^{-(2\ell+2)}$ (non-inclusive).
Let us verify that this gives a correct
colouring.

Take an edge $(x,y,z)$ and suppose both $x$ and $y$
get the same colour $2\ell$.   Then by the linear equation of the edge $v_z < -2^{-2\ell}$ and thus $z$ has a colour below $2\ell$.  The case
of two vertices of odd colour is similar and as the bound on the number of colours is immediate, the lemma follows.
\end{proof}

To find a solution which to apply \Cref{lemma:unbal} we first find a set of solutions $\{v^i\}_{i=1}^{n}$ to homogeneous system $Av=0^m$.  We require that the $i$th coordinate of $v^i$, i.e. $v^i_i$, equals $1/2$ and the maximum absolute value of any coordinate of $v^i$ is at most $1$. As $-3w/2$ or $3w/4$ satisfies these conditions such solution exists and some solution can be found by linear programming.  Define $u = \sum_{i=1}^{n} y_i v^i$ where $y_i$ are independent uniform variables in $[-1,1]$. 

\begin{lemma}
The vector $u$ has the following two properties:
\begin{enumerate}
    \item $\Pr[\min_i |u_i| \leq \frac{1}{4 n}] \leq \frac{1}{2}$

    \item $\Pr[\max_i |u_i| \geq 2 \sqrt{n \ln{n}}] \leq \frac{2}{n}$
\end{enumerate}
\end{lemma}

\begin{proof}
We prove the two bounds separately. 
\begin{enumerate}
\item We first show that $\Pr[|u_i| \leq \frac{1}{4n}] \leq \frac{1}{2n}$. A union bound, summing over all $i = 1, \ldots, n$, then implies that $\Pr[\min_i |u_i| \leq \frac{1}{4n}] \leq \frac{1}{2}$.

Observe that $u_i = \sum_j y_j v_i^j$, where $v^i_i = \frac{1}{2}$, $|v_i^j| \leq 1$ and $y_j$ is distributed uniformly at random in $[-1, 1]$. Suppose we sample $y_i$ last, and that the sum $\sum_{j \neq i} y_j v_i^j$ has evaluated to $\alpha$. Conditional on this fact, $u_i$ is uniformly distributed in $[\alpha - \frac{1}{2}, \alpha + \frac{1}{2}]$. The probability that $|u_i| \leq \frac{1}{4n}$ is then given by the length of the intersection of the interval $[\alpha - \frac{1}{2}, \alpha + \frac{1}{2}]$ with the interval $[- \frac{1}{4n}, \frac{1}{4n}]$ --- which is at most $\frac{1}{2n}$, as required.

\item This follows from standard Chernoff bounds as each coordinate of $u$ is the sum of $n$ independent random variables with mean $0$ and absolute value at most $1$.  The probability that such a variable takes the value at least $t$ is at most $\exp (-t^2/2n)$.  For a proof of this well known fact see appendix A of \cite{alonspencer}.  Applying this with $t=2 \sqrt{n \ln{n}}$ and using the union bound proves this part of the lemma. (The factor of 2 comes from the fact that we need to bound the \emph{absolute value} of $u_i$, not just $u_i$ itself.)\qedhere
\end{enumerate}
\end{proof}

Assuming that $n \geq 8$, with probability at least $\frac 14$, we can apply \Cref{lemma:unbal} with $M=2 \sqrt{n \ln{n}}$ and
$m= \frac{1}{4 n}$ and we conclude.

\begin{theorem}
Using the system of linear equations over the
rational numbers we can find, with probability at least $\frac{1}{4}$ and in polynomial time, an LO colouring with  $1.5 \log_2 n + O(\log \log n)$ colours.
\end{theorem}

This algorithm is less efficient compared to the algorithm of the previous section.  The main
computational cost is solving linear programs which is more complicated than solving linear systems of equations.  Our bound for the number of colours is
also worse.  It might be possible to find a random solution to the linear system in a different way but if an average coordinate has value $\Theta (1)$ then, heuristically, it seems likely that some of the $n$ coordinates would be at distance $O(1/n)$ from zero resulting in ratio $\Theta (n)$, between the highest and lowest absolute value.  Thus it seems unlikely that the algorithm based on the rational numbers would beat the algorithm described in the previous section. We suspect that it is possible to derandomise  also this algorithm by conditional expectations, but as this would be complicated let us ignore this possibility.

Nevertheless, we include this algorithm using $\mathbb{Q}$ because it is, in some sense, complementary to the first one. While the first one colours ``bottom-up'' (i.e.~always sets the minimal colour again and again), the algorithm using $\mathbb{Q}$ colours ``top-down'' (i.e.~sets the maximal colour again and again). It is interesting that the algorithm from~\cite{NZ23:toct} combined these two approaches, whereas these algorithms stick to only one approach each and get exponentially better results. The improvement seems to come from using the power of random solutions to the linear system $A$, whether solved over $\mathbb{Z}_2$ or $\mathbb{Q}$.

As a final observation in this section let us note that defining
a colouring by the sign of the vector $u$ we get a standard
(non-monochromatic) 2-colouring of the hypergraph.  This gives an alternative
algorithm to that of~\cite{BG21:sicomp,BG19}.

\section{Concluding remarks}

Our algorithms indicate that LO 2-colouring is quite different
from many other colouring problems.  The key property that we
use in our algorithm is that the constraint implies a linear
constraint.   The analysis of the algorithms also heavily
relies on the fact that we study 3-uniform hypergraphs.

It is tempting to think that the proposed methods would extend to
other constraint satisfaction problems where we are guaranteed
that a solution must satisfy a linear constraint.  We have so far been
unable to find an interesting such example.

\paragraph{Acknowledgement.} This paper is a merger of independent work by H{\aa}stad and Martinsson, and by Nakajima and {\v{Z}}ivn{\'{y}} respectively.
We are grateful to Venkat Guruswami for noting and informing the authors of the fact that we independently had found the same algorithm.
We thank  D\"om\"ot\"or P\'alv\"olgyi
for informing us that LO colourings have been studied under the name of \emph{unique-maximum} colourings.

{\small
\bibliographystyle{alphaurl}
\bibliography{bib}
}

\end{document}